\documentclass[11pt]{amsart}
\usepackage[foot]{amsaddr}
\usepackage[english]{babel}

\allowdisplaybreaks

\usepackage[
pdftex,hyperfootnotes]{hyperref}
\hypersetup{
	colorlinks=true,
	linkcolor=NavyBlue, 
	urlcolor=RoyalPurple,
	citecolor=OliveGreen,
	pdftitle={Generalized discrete multiple orthogonal polynomial},
	bookmarks=true,
	pdfpagemode=FullScreen,
}
\usepackage[usenames,dvipsnames,svgnames,table,x11names]{xcolor}
\usepackage{pgfplots}
\usepackage{tikz}
\usepackage{tikz-3dplot}
\usetikzlibrary{automata,quotes, chains,matrix,calc,shadows,shapes.callouts,shapes.geometric,shapes.misc,positioning,patterns,decorations.shapes,
	decorations.pathmorphing,decorations.markings,decorations.fractals,decorations.pathreplacing,shadings,fadings,arrows.meta,bending}

\usepackage{nicematrix}
\usepackage[utf8]{inputenc}

\usepackage{comment}
\usepackage[textwidth=17.5cm,textheight=22.275cm,
height=
24.275cm,
width=18cm]{geometry}

\usepackage{amssymb,latexsym,amsmath,amsthm,bm}
\usepackage{mathrsfs}
\usepackage{mathtools,arydshln,mathdots}
\mathtoolsset{showonlyrefs}

\usepackage[table]{xcolor}

\usepackage{Baskervaldx}
\usepackage[]{newtxmath}

\usepackage{bigints}

\theoremstyle{plain}

\newtheorem{teo}{Theorem}[section]

\newtheorem{lemma}[teo]{Lemma}
\newtheorem{pro}[teo]{Proposition}

\theoremstyle{defi}
\newtheorem{defi}[teo]{Definition}

\theoremstyle{remark}
\newtheorem{rem}[teo]{Remark}

\newcommand{\diag}{\operatorname{diag}}

\newcommand{\N}{\mathbb{N}}

\renewcommand{\intercal}{{\mathsf{\scriptscriptstyle T}}}

\allowdisplaybreaks[1]

\makeatletter
\DeclareRobustCommand{\gaussk}{\DOTSB\gaussk@\slimits@}
\newcommand{\gaussk@}{\mathop{\vphantom{\sum}\mathpalette\bigcal@{K}}}

\newcommand{\bigcal@}[2]{%
	\vcenter{\m@th
		\sbox\z@{$#1\sum$}%
		\dimen@=\dimexpr\ht\z@+\dp\z@
		\hbox{\resizebox{!}{0.8\dimen@}{$\mathcal{K}$}}%
	}%
}
\newcommand{\cfracplus}{\mathbin{\cfracplus@}}
\newcommand{\cfracplus@}{%
	\sbox\z@{$\dfrac{1}{1}$}%
	\sbox\tw@{$+$}%
	\raisebox{\dimexpr\dp\tw@-\dp\z@\relax}{$+$}%
}
\newcommand{\cfracdots}{\mathord{\cfracdots@}}
\newcommand{\cfracdots@}{%
	\sbox\z@{$\dfrac{1}{1}$}%
	\sbox\tw@{$+$}%
	\raisebox{\dimexpr\dp\tw@-\dp\z@\relax}{$\cdots$}%
}
\makeatother

\makeatletter
\newcommand*{\relrelbarsep}{.386ex}
\newcommand*{\relrelbar}{%
	\mathrel{%
		\mathpalette\@relrelbar\relrelbarsep
	}%
}
\newcommand*{\@relrelbar}[2]{%
	\raise#2\hbox to 0pt{$\m@th#1\relbar$\hss}%
	\lower#2\hbox{$\m@th#1\relbar$}%
}
\providecommand*{\rightrightarrowsfill@}{%
	\arrowfill@\relrelbar\relrelbar\rightrightarrows
}
\providecommand*{\leftleftarrowsfill@}{%
	\arrowfill@\leftleftarrows\relrelbar\relrelbar
}
\providecommand*{\xrightrightarrows}[2][]{%
	\ext@arrow 0359\rightrightarrowsfill@{#1}{#2}%
}
\providecommand*{\xleftleftarrows}[2][]{%
	\ext@arrow 3095\leftleftarrowsfill@{#1}{#2}%
}
\makeatother

\catcode`,\active

\catcode`\,12

\allowdisplaybreaks[1]
\begin{document}
\title[Discrete multiple orthogonality]{Toda  and Laguerre--Freud equations for multiple discrete  orthogonal polynomials with an arbitrary number of weights}

%
\author[I Fernández-Irisarri]{Itsaso Fernández-Irisarri$^1$}
\email{$^1$itsasofe@ucm.es}
\address{$^{1,2}$Departamento de Física Teórica, Universidad Complutense de Madrid, Plaza Ciencias 1, 28040-Madrid, Spain}

\author[M Mañas]{Manuel Mañas$^2$}
\email{$^2$manuel.manas@ucm.es}


\begin{abstract}
In this paper, we extend our investigation into semiclassical multiple discrete orthogonal polynomials by considering an arbitrary number of weights. We derive multiple versions of the Toda equations and the Laguerre-Freud equations for the multiple generalized Charlier and multiple generalized Meixner II families. 
\end{abstract}


\maketitle

\allowdisplaybreaks

\section{Introduction}

Discrete multiple orthogonal polynomials, $\tau$-functions, generalized hypergeometric series, Pearson equations, Laguerre--Freud equations, and Toda type equations are intricate mathematical concepts that are deeply intertwined and have broad applications across various fields of mathematics and physics. They are fundamental tools for understanding complex systems, integrable models, and specialized functions. 

In our prior work, cf. \cite{prior}, we focused on the scenario involving two weights. Now, in this exposition, we venture into a deeper investigation of discrete multiple orthogonal polynomials encompassing an arbitrary number of weights. We aim to unveil the intricate mathematical landscape surrounding these polynomials, uncovering their symbiotic connections with $\tau$-functions and generalized hypergeometric series. Additionally, we explore their fascinating applications in Laguerre--Freud equations and various instances of Toda equations.

Orthogonal polynomials, see \cite{Ismail, Chihara, Beals}, an indispensable cornerstone of Applied Mathematics and Theoretical Physics, hold sway over a vast array of disciplines, finding utility in approximation theory, numerical analysis, and beyond. Within this expansive realm, multiple orthogonal polynomials and discrete orthogonal polynomials emerge as particularly intriguing subclasses.

Originating in 1895, Pearson's pioneering work \cite{Pearson} in curve fitting gave birth to a renowned family of frequency curves encapsulated by the differential equation $\frac{f^{\prime}(x)}{f(x)}=\frac{p_1(x)}{p_2(x)}$. Here, $f$ denotes the probability density function, and $p_i$ represents polynomials of degree at most $i$, for $i=1,2$. Pearson equations, an enduring subject of study, find their nexus with orthogonal polynomials through their solutions, which often align with or closely resemble certain families of orthogonal polynomials. Central to this connection is the weight function that underpins the orthogonality properties of these polynomials.
These equations have garnered considerable attention owing to their intimate connection with special functions, including classical orthogonal polynomials like Legendre, Hermite, and Laguerre polynomials, as well as their extensions. Our approach in this study involves employing the Gauss-Borel method for orthogonal polynomials, extensively discussed in a review paper \cite{intro}. Building upon our prior research, which explored the application of this technique to various domains, such as discrete hypergeometric orthogonal polynomials, Toda equations, $\tau$-functions, and Laguerre--Freud equations  \cite{Manas_Fernandez-Irisarri,Fernandez-Irrisarri_Manas_2, Fernandez-Irrisarri_Manas_1.0, Fernandez-Irrisarri_Manas_1.1}, we delve deeper into its implications in this paper.

Multiple orthogonal polynomials \cite{nikishin_sorokin, Ismail, andrei_walter, afm} emerge when the orthogonality conditions involve multiple weight functions. In contrast with classical orthogonal polynomials, which rely on a single weight function, multiple orthogonal polynomials possess a more intricate structure due to the presence of several weight functions. These polynomials have been extensively studied and find applications in various fields, including simultaneous approximation, random matrix theory, and statistical mechanics. For insights into the utilization of the Pearson equation in multiple orthogonality, refer to \cite{Coussment, bfm} and also \cite{bfm2}. 

In contrast, discrete orthogonal polynomials \cite{Ismail, Nikiforov_Suslov_Uvarov} arise when orthogonality conditions are defined over a discrete set of points. These polynomials are employed in diverse areas such as combinatorial optimization, signal processing, and coding theory. They serve as valuable tools for analyzing and solving problems in discrete settings, where the underlying structure often comprises a finite or countable set of points. Discrete Pearson equations play a crucial role in the classification of classical discrete orthogonal polynomials \cite{Nikiforov_Suslov_Uvarov}. When the weight satisfies a discrete Pearson equation, we encounter semiclassical discrete orthogonal polynomials, as discussed in \cite{diego_paco, diego_paco1, diego, diego1}. Furthermore, multiple discrete orthogonal polynomials have been examined in \cite{Arvesu}.

The realm of integrable discrete equations \cite{Hietarinta} stands as an enthralling and consequential domain within Mathematical Physics, shedding light on the dynamics of discrete systems endowed with extraordinary attributes. These equations boast a complex algebraic and geometric framework, affording them soliton solutions, conservation laws, and integrability properties. Integrability in discrete equations transcends mere solvability, encompassing a wealth of conserved quantities and symmetries \cite{Adler}. This integrability property paves the way for the development of robust mathematical tools to analyze and comprehend their behavior.

A distinctive hallmark of integrable discrete equations lies in their intimate connection with orthogonal polynomials. These polynomials assume a pivotal role in crafting discrete equations with distinct characteristics. The nexus between discrete equations and orthogonal polynomials offers profound insights into their solutions, recursion relations, and symmetry properties.

The $\tau$-function, as elucidated by Harnad \cite{harnad}, stands as a linchpin in the domain of integrable systems, seamlessly bridging the spectral theory of linear operators with the dynamics of soliton equations. Its significance reverberates across various mathematical realms, including algebraic geometry, representation theory, and special functions.

Semiclassical orthogonal polynomials and their associated functionals, supported on complex curves, have been subject to extensive scrutiny in the context of isomonodromic $\tau$ functions and matrix models by Bertola, Eynard, and Harnad \cite{bertola}. 

In the sphere of discrete multiple orthogonal polynomials, the $\tau$-function assumes a paramount role, encapsulating the intrinsic integrable structure and offering insights into the dynamics and symmetries of the corresponding discrete systems. Its exploration unveils profound connections between diverse mathematical entities and unveils hidden patterns lurking within.

Generalized hypergeometric series, a focal point of this research, represent a class of special functions prevalent across diverse mathematical landscapes \cite{Hipergeometricos, LibrodeHypergeom}. Renowned for their exceptional properties, they have been subject to extensive inquiry owing to their connections with combinatorics, number theory, and mathematical physics. These series play a pivotal role in expressing solutions of differential equations, evaluating integrals, and comprehending the behavior of various mathematical models.

Crucially, generalized hypergeometric series occupy a central position in the Askey scheme, a comprehensive framework that unifies essential families of orthogonal functions through a chain of limits \cite{Koekoek}. The Askey scheme, as extended to encompass multiple orthogonality in \cite{AskeyII}, provides a structured approach to understanding various families of special functions.

Recently, we have made strides in advancing the multiple Askey scheme by establishing hypergeometric expressions for Hahn type I multiple orthogonal polynomials and their progeny within this scheme \cite{BDFM}, see \cite{BDFM1} for the corresponding bidiagonal factorization, and \cite{BDFM2,BDFM3} for the explicit expressions for type I and II classical multiple orthogonal polynomials and ts recursion coefficients.

Laguerre--Freud equations form a family of differential equations that surface in the realm of orthogonal polynomials associated with exponential weights. These equations encode nonlinear connections among the recursion coefficients within the three-term recurrence relation for orthogonal polynomials. In essence, they entail a set of nonlinear equations governing the coefficients $\{\beta_n,\gamma_n\}$ within the recurrence relation $zP_n(z) = P_{n+1}(z) + \beta_nP_n(z) + \gamma_nP_{n-1}(z)$ characterizing the orthogonal polynomial sequence. These equations manifest as $\gamma_{n+1} = \mathcal{G}(n, \gamma_n, \gamma_{n-1}, \dots, \beta_n, \beta_{n-1},\dots)$ and $\beta_{n+1} = \mathcal{B}(n,\gamma_{n+1},\gamma_n,\dots,\beta_n,\beta_{n-1},\dots)$.

Named Laguerre--Freud equations by Magnus \cite{magnus,magnus1,magnus2,magnus3} with reference to \cite{laguerre,freud}, these equations have garnered attention in various contexts. Numerous studies delve into Laguerre--Freud relations concerning generalized Charlier, generalized Meixner, and type I generalized Hahn cases \cite{smet_vanassche,clarkson,filipuk_vanassche0,filipuk_vanassche1,filipuk_vanassche2,diego}.

Laguerre--Freud equations share a profound connection with Painlevé equations, a family of nonlinear ordinary differential equations initially introduced by  Painlevé in the early 20th century. These equations have garnered considerable attention owing to their integrability properties and the intricate mathematical structures underpinning their solutions. 

The intertwining of orthogonal polynomials with Painlevé equations stems from the fact that specific families of orthogonal polynomials can be linked to solutions of particular Painlevé equations. This association offers profound insights into the solvability and analytic properties of Painlevé equations, enabling the construction of explicit solutions through the theory of orthogonal polynomials. For a comprehensive exploration of this interplay between these domains, refer to \cite{Van Assche,clarkson,clarkson2}.

Expanding upon the insights gleaned from prior investigations \cite{Manas_Fernandez-Irisarri,Fernandez-Irrisarri_Manas_1.0,Fernandez-Irrisarri_Manas_1.1, Fernandez-Irrisarri_Manas_2}, this paper delves into the  interrelations among discrete multiple orthogonal polynomials, $\tau$-functions, generalized hypergeometric series, and specific equations such as Laguerre--Freud equations and Toda type equations.

In contrast, Toda type equations represent nonlinear partial differential-difference equations of paramount importance in mathematical physics and integrable systems \cite{harnad,Hietarinta}. Unraveling the nexus between these equations and the aforementioned mathematical constructs offers profound insights into the dynamics of discrete systems, the behavior of special functions, and the intricacies of integrable models.

In this paper, our goal is to offer a comprehensive exploration of discrete multiple orthogonal polynomials with an arbitrary number of weights, $\tau$-functions, generalized hypergeometric series, Laguerre--Freud equations, and Toda type equations. 
Moreover, we will investigate the intricate interplay between these entities, uncovering their connections and mutual dependencies. Specifically, we will examine how discrete multiple orthogonal polynomials intertwine with $\tau$-functions and generalized hypergeometric series, unraveling the underlying mathematical relationships.

Furthermore, we will delve into the practical implications of these mathematical concepts by exploring their applications in Laguerre--Freud equations and Toda type equations. By examining both continuous and discrete variants, we aim to showcase their significance in the realm of mathematical physics and integrable systems, illustrating their versatility and utility in diverse contexts.

The layout of the paper is as follows. We continue this section with a basic introduction to discrete multiple orthogonal polynomials and their associated tau functions. In the second section, we delve into Pearson equations for multiple orthogonal polynomials, generalized hypergeometric series, and the corresponding hypergeometric discrete multiple orthogonal polynomials. We discuss their properties, including the Laguerre--Freud matrix that models the shift in the spectral variable, as well as contiguity relations and their consequences. 

In the third section, we present two of the main findings discussed in this paper. In Theorem \ref{theo:third_order_PDE}, we present nonlinear partial differential equations of the Toda type, for which the hypergeometric tau functions provide solutions. Additionally, we discuss  a completely discrete system, extending the completely discrete Nikhoff--Capel Toda equation, and its solutions in terms of the hypergeometric tau functions. 

Finally, in the fourth section, we present explicit Laguerre--Freud equations for the first time for the multiple versions of the generalized Charlier and generalized Meixner II orthogonal polynomials.

\subsection{Discrete multiple orthogonality on the step line with an arbitrary number of weights}

Let's delve into the fundamentals of discrete multiple orthogonality concisely.

We start by defining two sets of vectors of monomials in the variable $x$:
\[
\begin{aligned}
	X &= \begin{bNiceMatrix} 1 \\ x \\ x^2 \\ \Vdots \end{bNiceMatrix}, &
	X^{(a)} &= \begin{bNiceMatrix} e_a \\ x e_a \\ x^2 e_a \\ \Vdots \end{bNiceMatrix},
\end{aligned}
\]
where $a \in \{1, \dots, p\}$ and $\{e_a\}_{a=1}^p$ represents the canonical basis in $\mathbb{R}^p$.

Next, we introduce the moment matrix:
\[
\mathscr{M} \coloneqq \sum^{\infty}_{k=0} X(k)\big(X^{(1)}(k)w^{(1)}(k)+\dots+X^{(q)}(k)w^{(q)}(k)\big)^{\intercal},
\]
where functions $w^{(a)}$, $a \in \{1, \dots, p\}$, denote $p$ weights supported on the homogeneous lattice $\mathbb{N}_0$.

Now, let's discuss some  key matrix components:
\begin{enumerate}
	\item \textbf{\emph{Projection Matrices:}} These are semi-infinite matrices defined as
\[
I^{(a)} \coloneqq \text{diag}(0, \dots, 1, 0, \dots, 0, 1, \dots),
\]
where $a \in \{1, \dots, p\}$, and non-zero diagonal entries occur in positions $a+ip$ for $i \in \mathbb{N}$.
\item \textbf{\emph{Truncated Projection Matrices:}} These are defined as
\[
\tilde{I}^{(a)} \coloneqq \text{diag}(0, \dots, 0, 1, 0, \dots, 0) \in \mathbb{R}^{p\times p},
\]
with the non-zero element at $(\tilde{I}^{(a)})_{a,a}$, for $a\in\{1,\dots,p\}$.
\item \textbf{\emph{Shift Matrices: }}These are given by
\[
\Lambda \coloneq  \left[\begin{NiceMatrix}[columns-width=auto]
		0 & 1 & 0&\Cdots&\\
		0& 0&1 &\Ddots&\\
		\Vdots[shorten-end=4pt]&\Ddots[shorten-end=-10pt]&\Ddots& \Ddots&\\
		&&& &
	\end{NiceMatrix}\right],
\]
and $\Lambda^{(a)} \coloneqq \Lambda^q I^{(a)}$, $a \in \{1, \dots, p\}$.
\end{enumerate}

We observe the following equations hold true:
\[
\begin{aligned}
	\Lambda X(x) &= xX(x), & \Lambda^{(a)}X^{(b)}(x) &= xX^{(a)}(x)\delta_{a,b}.
\end{aligned}
\]
The moment matrix $\mathscr{M}$ is a $p$-Hankel matrix, satisfying:
\[
\Lambda \mathscr{M} = \mathscr{M}(\Lambda^{\intercal})^p.
\]

We then consider the Gauss-Borel factorization for $\mathscr{M}$:
\[
\mathscr{M} = S^{-1}H\tilde{S}^{-\intercal},
\]
where $S$ and $\tilde{S}$ are lower unitriangular matrices and $H$ is diagonal.

These matrices $S$ and $\tilde{S}$ can be expressed by adding diagonal matrices:
\begin{align*}
\begin{aligned}
	S &= I + \Lambda^{\intercal} S^{[1]} + (\Lambda^{\intercal})^2 S^{[2]} + \cdots, & S^{-1} &= I + \Lambda^{\intercal} S^{[-1]} + (\Lambda^{\intercal})^2 S^{[-2]} + \cdots,
\end{aligned}\\
\begin{aligned}
	\tilde{S} &= I + \Lambda^{\intercal} \tilde{S}^{[1]} + (\Lambda^{\intercal})^2 \tilde{S}^{[2]} + \cdots, & \tilde{S}^{-1} &= I + \Lambda^{\intercal} \tilde{S}^{[-1]} + (\Lambda^{\intercal})^2 \tilde{S}^{[-2]} + \cdots.
\end{aligned}
\end{align*}

Additionally, we define the lowering and rising shift operators, $\mathfrak{a}_{-}$ and $\mathfrak{a}_{+}$, for diagonal matrices $D$ as:

\[
\begin{aligned}
	\mathfrak{a}_{-}\text{diag}(m_0,m_1,\dots) &= \text{diag}(m_1,m_2,\dots), & \mathfrak{a}_{+}\text{diag}(m_0,m_1,\dots) &= \text{diag}(0,m_0,\dots).
\end{aligned}
\]

Lastly, we define multiple orthogonal polynomials:
\begin{enumerate}
\item \textbf{\emph{Type I: }}
\[
A^{(n)} \coloneqq H^{-1}\tilde{S}X^{(n)},
\]
where $n \in \{1, \dots, p\}$.
\item \textbf{\emph{Type II:}}
\[
B \coloneqq SX.
\]
\end{enumerate}

For perfect systems of weights, the degrees of these polynomials are:
\[
\text{deg } B_n = n, \quad \text{deg } A_n^{(a)} = \left\lceil\frac{n+2-a}{p}\right\rceil-1, \quad a \in \{1, \dots, p\},
\]
where $\lceil x\rceil$ represents the smallest integer greater than or equal to $x$. Multiple orthogonality relations are expressed as:
\[
\begin{aligned}
	\sum_{a=1}^p \sum_{k=0}^{\infty} A_n^{(a)}(k) w^{(a)}(k) k^m& = 0, & k& \in \{0, \ldots, \text{deg } B_{n-1}\}, \\
	\sum_{k=0}^{\infty} B_n(k) w^{(a)}(k) k^m &= 0, & k &\in \{0, \ldots, \text{deg } A_{n-1}^{(a)}\}, & a& \in \{1, \dots, p\}.
\end{aligned}
\]
Let's introduce the $\tau$-functions:
\begin{enumerate}
	\item $\tau$ functions are defined departing from truncations of the moment matrix $M^{[n]}$ as:
	\begin{equation}
		\tau_n \coloneq \det \mathscr{M}^{[n]}.
	\end{equation}
	
	\item $\tau$-associated functions $\tau^j_n$ with $j\in\{1, \dots, n-1\}$, are the matrix determinants $\mathscr{M}^{[n,j]}$ obtained from $\mathscr{M}^{[n]}$ by removing the $(n-j)$-th row $\begin{bmatrix} M_{n-j,0} & \dots & M_{n-j,n-1} \end{bmatrix}$ and adding, as the last row, the row $\begin{bmatrix} M_{n,0} & \dots & M_{n,n-1} \end{bmatrix}$.
\end{enumerate}
  
  Type I polynomials can be expressed as:
  \begin{align}\label{eq:determinant_A}
  	\begin{aligned}
  		A^{(a)}_n(x) &= \frac{1}{\tau_{n+1}}\begin{vNiceArray}{cw{c}{1cm}c|c}[margin]
  			\Block{3-3}<\Large>{\mathscr M^{[n]}} & & & \mathscr{M}_{0,n} \\
  			& & & \Vdots \\
  			&&& \mathscr M_{n-1,n}\\
  			\hline
  			X^{(a)}_0 & \Cdots & X^{(a)}_{n-1} & X^{(a)}_n
  		\end{vNiceArray}, & n\in &\N_0, & & a&\in\{1,\dots,p\}.
  	\end{aligned}
  \end{align}
  
  Similarly, type II polynomials can be expressed as:
  \begin{align}\label{eq:determinant_B}
  	\begin{aligned}
  		B_n(x) &= \frac{1}{\tau_n}\begin{vNiceArray}{cw{c}{1cm}c|c}[margin]
  			\Block{3-3}<\Large>{\mathscr M^{[n]}} & & & 1 \\
  			& & & \Vdots \\
  			&&& x^{n-1}\\
  			\hline
  			\mathscr{M}_{n,0} & \Cdots & \mathscr M_{n,n-1} & x^n
  		\end{vNiceArray}, & n\in \N_0.
  	\end{aligned}
  \end{align}
  
  These multiple orthogonal polynomials satisfy recurrence relations. Considering the multi-Hankel structure of the moment matrix $\mathscr M$, the banded Hessenberg matrix $T$ can be constructed in terms of $S$ and $\tilde{S}$ as follows:
  \[
  T \coloneq S\Lambda S^{-1} = H\tilde{S}^{-\intercal}(\Lambda^\intercal)^p\tilde{S}^{\intercal}H^{-1}.
  \] 
  Thus, we have a $p+2$ terms recurrence relation that can be written as: 
  \[
  \begin{aligned}
  	TB(x) &= xB(x), & T^\intercal A^{(a)}(x) &= xA^{(a)}(x).
  \end{aligned}
  \]
  The banded Hessenberg matrix has $p+2$ non-zero diagonals:
  \[
  T = \Lambda + \alpha^{(0)} + \Lambda^{\intercal}\alpha^{(1)} + \dots + (\Lambda^\intercal)^q\alpha^{(q)},
  \]
  where expressions for $\alpha$ matrices depending on $\tilde{S}$ or $S$ are as follows:
  \begin{align*}
  	\alpha^{(0)} &= \mathfrak{a}_{+}S^{[1]}-S^{[1]}, \\
  	\alpha^{(1)} &= \mathfrak{a}_{+}S^{[2]}-S^{[2]}+S^{[1]}(\mathfrak{a}_{-}S^{[1]}-S^{[1]}), \\
  	\alpha^{(a)} &= \mathfrak{a}_{+}S^{[a+1]}+S^{[-a-1]}+\sum^{a}_{i=1}\mathfrak{a}_{-}^aS^{[a-i]}S^{[-i]}, \\
  	&  = H^{-1}(\mathfrak{a}_{-}^aH)\biggl(\mathfrak{a}_{+}^{p-a}\tilde{S}^{[p-a]}+\mathfrak{a}_{-}^a\tilde{S}^{[-(p-a)]} +\sum^{q-a-1}_{j=1}(\mathfrak{a}_{+}^j\tilde{S}^{[j]})\mathfrak{a}_{-}^a\tilde{S}^{[-(p-a-j)]}\biggr), \\
  	&\hspace{7pt}\vdots \\
  	\alpha^{(p-1)} &= H^{-1}(\mathfrak{a}_{-}^{p-1}H)(\mathfrak{a}_{+}\tilde{S}^{[1]}-\mathfrak{a}_{-}^{q-1}\tilde{S}^{[1]}), \\
  	\alpha^{(p)} &= (\mathfrak{a}_{-}^pH)H^{-1}.
  \end{align*}

    Partial recurrence matrices  $T^{(n)}$ are given  as follows
    \begin{equation*}
        T^{(a)}=H^{-1}\tilde{S}\Lambda_a\tilde{S}^{-1}H.
    \end{equation*}
Note  that 
    \begin{equation*}
     \sum^{q}_{a=1}T^{(a)}=T^\intercal.   
    \end{equation*}

 For this discrete scenario we require also of    the Pascal matrices $L^{\pm}$ are defined as matrices with entries are   
    \begin{align*} 
L^{+}_{n,m}&=\begin{cases} \binom{n}{m},& n \geq m, \\ 0,& n<m,  \end{cases}&
L^{-}_{n,m}&=\begin{cases} {}(-1)^{n+m} \binom{n}{m},& n \geq m, \\ 0&n<m.\end{cases}
\end{align*}
Note that $L^+L^-=I.$
Moreover, we can readily check  that
\begin{align*}
    X(x+1)&=L^+X(x), & X(x-1)&=L^-X(x).
\end{align*}
Pascal matrices can be expanded as follows
\begin{equation*}
    L^{\pm}=I\pm\Lambda^\intercal D+(\Lambda^\intercal)^2D^{[2]}+\cdots,
\end{equation*}
where 
\[\begin{aligned}
    D&\coloneq\diag(1,2,3,\dots), & D^{[n]}&\coloneq\frac{1}{n}\diag(n^{(n)},(n+1)^{(n)},\dots)
\end{aligned}\]
and $x^{(n)}=x(x-1)\dots(x-n+1).$

Partial Pascal matrices $L^{(n) \pm}$ are defined departing from the truncated projectors as follows and they have the following nonzero entries:
\[
L_{2 n+a-1,2 m+a-1}^{ \pm(a)}=( \pm 1)^{n+m}\left(\begin{array}{c}
	n \\
	m
\end{array}\right), \quad a \in\{1,\dots,p\}, \quad n, m \in \mathbb{N}_0, \quad n \geq m
\]

Partial Pascal matrices  satisfy
\[\begin{aligned}  X^{(a)}(x+1)&=L^{(a) +}X^{(a)}(x), & X^{(a)}(x-1)&=L^{(a)-}X^{(a)}(x).\end{aligned}\]
Partial Pascal matrices can be expanded as:
\begin{equation*}
    L^{(a)\pm}=I^{(a)}\pm(\Lambda^\top)^qD_a+(\Lambda^\top)^{2q}D_a^{[2]}+\dots.
\end{equation*}
where we have
$$
\begin{aligned}
	D_a^{[n]}&\coloneq \frac{1}{n} \operatorname{diag}\left(n^{(n)}e_a, (n+1)^{(n)}e_a,  \ldots\right),  & a&\in\{1,\dots,p\}.
\end{aligned}
$$
 The dressed Pascal matrices can be defined from Pascal matrices as follows
\[ \begin{aligned} \Pi^{+}&=SL^{+}S^{-1}, & \Pi^{-}&=SL^{-}S^{-1}\end{aligned}\]
and satisfy
\[\begin{aligned}B(x+1)&=\Pi^+B(x), & B(x-1)&=\Pi^-B(x).\end{aligned}\]

    Partial dressed Pascal matrices are defined as follows
\[     \begin{aligned} \Pi^{(a)+}&=H^{-1}\tilde{S}L^{(a)+}\tilde{S}^{-1}H, & \Pi^{(a)-}&=H^{-1}\tilde{S}L^{(a)-}\tilde{S}^{-1}H,\end{aligned}\]
so that  the following relations hold
\[\begin{aligned}
    \Pi^{(a)+}A^{(a)}(x)&=A^{(a)}(x+1), & \Pi^{(a)-}A^{(a)}(x)&=A^{(a)}(x-1).
\end{aligned}\]

In this paper, generalized hypergeometric series are used, and they are given by
    \begin{equation*}
        {}_MF_N\begin{pmatrix}b_1,\dots,b_M&\\ &;\eta\\ c_1,\dots,c_N& \end{pmatrix}= {}_MF_N(b_1,\dots,b_M; c_1,\dots, c_N;\eta):=\sum^{\infty}_{k=0}\frac{(b_1)_k\dots(b_M)_k}{(c_1)_k\dots(c_N)_k}\frac{\eta^k}{k!}
    \end{equation*}
    where $(a)_k$ is known as the Pochhammer symbol and is defined as
\[    \begin{aligned}
        (a)_k&:=a(a+1)\dots(a+k-1), & (a)_0&=1.
    \end{aligned}\]

\section{Discrete Multiple Orthogonal Polynomials and Pearson equation}

\subsection{Pearson equations}
From hereon we will assume that all the weights adhere to the Pearson equations, represented as:
\begin{equation}\label{eq:pearson}
	\theta(x+1)w^{(a)}(x+1)=\sigma_n(x)w^{(a)}(x),\end{equation}
where $\theta(x)$ and $\sigma_a(x)$ are polynomials and $a\in\{1,\dots,p\}$.

First, we give an important relation fulfilled by the moment matrix $\mathscr M$:
\begin{teo}
	The moment matrix satisfies the following equation:
	\begin{equation}\label{symMom}
		\theta(\Lambda)\mathscr M=L^+\mathscr M\left(\sum_{a=1}^pL^{(a)+}\sigma_a(\Lambda^{(n)})\right)^\top.
	\end{equation}
\end{teo}

\begin{proof}
	Beginning with $\theta (\Lambda)\mathscr M$, we expand it as follows:
	\begin{align*}
		\theta(\Lambda)\mathscr M&=\sum^{\infty}_{k=1}\theta(k)X(k)\left(\sum^{p}_{a=1}w^{(a)}(k)X^{(a)}(k)\right)^\top\\
		&=\sum^{\infty}_{k=0}\theta(k+1)X(k+1)\left(\sum^{p}_{a=1}w^{(a)}(k+1)X^{(a)}(k+1)\right)^\top\\
		&=\sum^{\infty}_{k=0}X(k+1)\left(\sum^{p}_{a=1}\sigma_a(k)w^{(a)}(k)X^{(a)}(k+1)\right)^\top\\
		&=\sum^{\infty}_{k=0}L^+X(k)\left(\sum^p_{a=1}\sigma_a(k)w^{(a)}(k)X^{(a)}(k)^\top(L^{(a)+})\right)^\top\\
		&=L^+\mathscr M\left(\sum^p_{a=1}L^{(a)+}\sigma_a(\Lambda^{(a)})\right)^\top.
	\end{align*}
\end{proof}
Then, we provide its counterpart for the Hessenberg recurrence matrix:
\begin{pro}
	The following equation holds for the Hessenberg matrix \( T \) if the weights satisfy the Pearson equations:
	\begin{equation}
		\Pi^{-}\theta(T)=\sum^p_{a=1}\sigma_a(T^{(a)})^\top(\Pi^{(a)+})^\top.
	\end{equation}
\end{pro}
\begin{proof}
	Utilizing equation \eqref{symMom} and the Gauss-Borel factorization for the moment matrix, we have
	\begin{equation*}
		\theta(\Lambda)S^{-1}H\tilde{S}^{-\top}=L^+S^{-1}H\tilde{S}^{-\top}\left(\sum^{p}_{a=1}L^{(a)+}\sigma_a(\Lambda^{(a)})\right)^\top.
	\end{equation*}
	If we multiply both sides by \( S \) on the right and by \( \tilde{S}^\top H^{-1} \) on the left, we obtain:
	\begin{equation*}
		S\theta(\Lambda)S^{-1}=SL^+S^{-1}H\tilde{S}^{-\top}\sum^p_{a=1}\sigma_a(\Lambda^{(a)})^\top\tilde{S}^\top H^{-1}H\tilde{S}^{-\top}(L^{(a)+})^\top\left( H\tilde{S}^{-\top}\right)^{-1}.
	\end{equation*}
	Finally, we get
	\begin{equation*}
		\theta(T)=\Pi^+\sum^p_{a=1}\sigma_a(T^{(a)})^\top(\Pi^{(a)+})^\top.
	\end{equation*}
\end{proof}
\subsection{Relation with generalized hypergeometric series}
If we write the polynomials $\theta$ and $\sigma_n$ as 
\[\begin{aligned}
    \theta(k)&=k(k+c_1-1)\cdots(k+c_N-1), &  \sigma_a(k)=\eta^{(a)}(k+b_1^{(a)})\cdots(z+b_{M^{(a)}}^{(a)}),
\end{aligned}\]
with $M^{(a)},N \in \mathbb{N}_0$ and $a \in \{1,\dots,p\},$  the weights that satisfy the Pearson equation are:
\begin{equation*}
    w^{(a)}(k)=\frac{(b_1^{(m)})_k\cdots(b_{M^{(a)}}^{(a)})_k}{(c_1)_k\dots(c_N)_k}\frac{(\eta^{(a)})^k}{k!},
\end{equation*}
where $a\in \{1,\dots,p\}.$
We define the following operators:
\[\begin{aligned}
    \vartheta^{(a)}&:=\eta^{(a)}\frac{\partial}{\partial \eta^{(a)}}, &   \vartheta&:=\sum^{q}_{a=1}\vartheta^{(a)}.
\end{aligned}\]
\begin{teo}\label{teo:moments}
    The moment matrix can be written in terms of hypergeometric functions as follows
    \begin{equation}
        \mathscr M=\left[\begin{NiceMatrix}
            \rho^{(1)}_0&\rho^{(2)}_0&\Cdots&\rho^{(p)}_0&\rho^{(1)}_1&\Cdots&\rho^{(p)}_1&\Cdots\\[3pt]
            \rho^{(1)}_1&\rho^{(2)}_1&\Cdots&\rho^{(p)}_1&\rho^{(1)}_2&\Cdots&\rho^{(p)}_2&\Cdots\\
            \Vdots&\Vdots&&\Vdots&\Vdots&&\Vdots&
        \end{NiceMatrix}\right],
    \end{equation}
    where the moments are
    \begin{align}\label{moments}\rho^{(a)}_0&:={}_{M_a}F_N(b_1^{(a)},\dots,b^{(a)}_{M^{(a)}};c_1,\dots,c_N;\eta^{(a)}), &       \rho^{(a)}_m&=(\vartheta^{(a)})^m\rho^{(a)}_0,\end{align}
    and $a\in \{1,\dots,p\}.$
\end{teo}
\begin{proof}
    It follows from the definition of the moment matrix.
\end{proof}
\begin{defi}
    We introduce the notation
    \begin{align*}
    	\vec f&=\begin{bNiceMatrix}
    		f_1&\Cdots &f_p
    	\end{bNiceMatrix}, & 	\vec f^{[0]}&=\emptysetAlt,&\begin{aligned}
    	\vec f^{[r]}&=\begin{bNiceMatrix}
    	f_1&\Cdots &f_r
    	\end{bNiceMatrix}, &r&\in\{1,\dots,p-1\}.\end{aligned}
    \end{align*}
    The extended Wronskians are the following determinants:

    \NiceMatrixOptions{cell-space-limits = 1pt}\begin{align*}
    			\mathscr W_{pn+r}[\vec f]&\coloneq \begin{vNiceArray}{w{c}{42pt}|w{c}{42pt}|w{c}{42pt}|ccc|w{c}{52pt}|w{c}{62pt}}[margin]
    			\vec f  &\vartheta \vec f& \vartheta^2\vec f& \phantom{t}&\Cdots[shorten-start=-5pt,shorten-end=-5pt]&\phantom{t}&\vartheta^{n-1}\vec f&\vartheta^{n} \vec f^{[r]}\\\hline
    			\vartheta\vec f  &\vartheta^2 \vec f& \vartheta^3\vec f& \phantom{t}&\Cdots[shorten-start=-5pt,shorten-end=-5pt]&\phantom{t}&\vartheta^{n}\vec f&\vartheta^{n+1} f^{[r]}\\\hline
    			\Vdots&		\Vdots&		\Vdots&&		&&\Vdots&\Vdots\\\hline
    			\vartheta^{pn-1}\vec f  &\vartheta^{pn} \vec f& \vartheta^{pn+1}\vec f& \phantom{t}&\Cdots[shorten-start=-5pt,shorten-end=-5pt]&\phantom{t}&\vartheta^{(p+1)n-2}\vec f &\vartheta^{(p+1)n-1} f^{[r]}\\\hline
    			\Vdots&		\Vdots&		\Vdots&&		&&\Vdots&\Vdots
    			\\\hline
    			\vartheta^{pn+r-1}\vec f  &\vartheta^{pn+r} \vec f& \vartheta^{pnr+1}\vec f& \phantom{t}&\Cdots[shorten-start=-5pt,shorten-end=-5pt]&\phantom{t}&\vartheta^{(p+1)n+r-2}\vec f &\vartheta^{(p+1)n+r-1} f^{[r]}
    	    		\end{vNiceArray},
    	\end{align*}

\end{defi}
\begin{pro}
    The $\tau$ function is the extended Wronskian of  a hipergeometric series:
    \begin{equation}
        \tau_n=\mathscr W_n[{}_{M_1}F_N(b_1^{(1)},\dots,b^{(1)}_{M^{(1)}};c_1,\dots,c_N;\eta^{(1)}),\dots,{}_{M_p}F_N(b_1^{(p)},\dots,b^{(p)}_{M^{(p)}};c_1,\dots,c_N;\eta^{(p)})].
    \end{equation}
\end{pro}
\begin{proof}
	It follows from Theorem \ref{teo:moments}.
\end{proof}
\begin{pro}
    The following equation holds
    \begin{equation}
        \vartheta^{(a)}\mathscr M=\mathscr M(\Lambda^{(a)})^\top.
    \end{equation}
\end{pro}
\begin{proof}
    It follows from \eqref{moments}.
\end{proof}

\subsection{Laguerre--Freud Matrix}
Next, we demonstrate that when a Pearson equation is fulfilled, there is a banded semi-infinite matrix that connects the shifted multiple orthogonal polynomials in the homogeneous lattice with the non-shifted ones; i.e., it models translations in the lattice.
\begin{teo}
Suppose the weights adhere to Pearson equations \eqref{eq:pearson}, where $\theta$ and $\sigma_1,\dots,\sigma_q$ are polynomials satisfying $\deg{\theta}=N$ and $\max\{\deg{\sigma_1},\dots,\deg{\sigma_q}\}=M$. Then, we establish the following relationship:
\begin{equation}\label{L-F}
	\Pi^-\theta(T)=\sum^p_{a=1}\sigma_a(T^{(a)})^\top(\Pi^{(a)+})^\top\eqcolon \Psi.
\end{equation}
This matrix $\Psi$ emerges as the Laguerre--Freud matrix. It exhibits a banded structure, characterized by \(pM\) subdiagonals and \(N\) superdiagonals, detailed as follows:
\begin{equation}\label{L-F diag}
	\Psi=(\Lambda^\top)^{pM}\psi^{(-pM)}+\dots+\psi^{(N)}\Lambda^N.
\end{equation}
\end{teo}
\begin{proof}
	The relation \eqref{L-F} follows from the symmetry equation \eqref{symMom} and the Gauss--Borel factorization.
	The equation \eqref{L-F diag} follows from the degree of $\theta$ and maximun of the degrees of $\sigma$, that departing from \eqref{L-F} gives us the number of non-zero superdiagonals and non-zero subdiagonals, respectively.
\end{proof}
\begin{teo}
For $a\in\{1,\dots,p\}$,   the following equations are fulfilled
    \begin{align}
      \label{conn1}  \theta(x)B(x-1)&=\Psi B(x), &\Psi^\top A^{(a)}(x)&=\sigma_a(x)A^{(a)}(x+1).
    \end{align}
\end{teo}
\begin{proof}
    For the type I polynomials we can write:
    \begin{equation*}
        \Psi^\top A^{(a)}(x)=\sum^q_{a=1}\Pi^{(a)+}\sigma_a(T^{(a)})A^{(a)}(x)=\Pi^{(a)+}\sigma_a(x)A^{(a)}(x)=\sigma_a(x)A^{(a)}(x+1).
    \end{equation*}
    For the type II polynomials we have that
    \begin{equation*}
        \Psi B(x)=\Pi^-\theta(T)B(x)=\Pi^-\theta(x)B(x)=\theta(x)B(x-1).
    \end{equation*}
\end{proof}
The Laguerre--Freud and the recurrence matrices are closely connected:
\begin{pro}
    The following compatibility formula holds
    \begin{equation}
        [\Psi,T]=\Psi.
    \end{equation}
\end{pro}
\begin{proof}
    Using \eqref{conn1} we obtain the following two equations:
    \begin{equation*}
        T\Psi B(x)=T\theta(x)B(x-1)=\theta(x)(x-1)B(x-1)
    \end{equation*}
    and
    \begin{equation*}
        \Psi TB(x)=\Psi xB(x)=x\theta(x)B(x-1),
    \end{equation*}
    with both equations above we have that
    \begin{equation*}
        [\Psi,T]B(x)=\theta(x)B(x-1)=\Psi B(x).
    \end{equation*}
\end{proof}
\subsection{Contiguous hypergeometric relations and connection matrices}
Generalized hypergeometric series fulfill:
\begin{equation*}
    \biggl(\eta\frac{d}{d\eta}+b_i\biggr){}_MF_N(b_1,\dots,b_M;c_1,\dots,c_N;\eta)=b_i{}_MF_N(b_1,\dots,b_i+1,\dots,b_M;c_1,\dots,C_N;\eta)
\end{equation*}
and
\begin{equation*}
    \biggl(\eta\frac{d}{d\eta}+c_j-1\biggr){}_MF_N(b_1,\dots,b_M;c_1,\dots,c_N;\eta)=b_i{}_MF_N(b_1,\dots,b_M;c_1,\dots,c_j-1,\dots,C_N;\eta)
\end{equation*}
for $i\in\{1,\dots,M\}$ and $j\in\{1,\dots,N\}.$

We will use the following notation:    ${}_i\Theta^{(a)}$ corresponds to the shift $b_i^{(a)} \rightarrow b_i^{(a)}+1$ and $\Theta_j$ corresponds to the shift $c_j \rightarrow c_j-1$.

These relations imply for the moment matrix:
\begin{teo}
   For $a\in\{1,\dots,p\}$, the following equations for the moment matrix are fulfilled
    \begin{align}
       \label{hypRelM1} \mathscr M(\Lambda^{(a)})^\top+b_i^{(a)}\mathscr M=b_i^{(a)}{}_i\Theta^{(a)}\mathscr M,\\
       \label{hypRelM2} \mathscr M(\Lambda^q)^\top+(c_j-1)\mathscr M=(c_j-1)\Theta_j\mathscr M.
    \end{align}
\end{teo}
\begin{proof}
The equations \eqref{hypRelM1} and \eqref{hypRelM2} are derived from the contiguous hypergeometric relations for generalized hypergeometric functions. To prove \eqref{hypRelM1}, we first analyze how the operator \((\vartheta_{\eta^{(a)}} + b_i^{(a)})\) acts on different entries of the moment matrix. There are two scenarios: one where it acts on \(\rho^a_m\) and another where it acts on \(\rho^k_m\) with \(k \neq a\).

Firstly, we have:
\[
(\vartheta_{\eta_a} + b_i^{(a)})\rho^a_m = (\vartheta_{\eta_a} + b_i^{(a)})(\vartheta_{\eta_a}^m \rho_0^a) = \vartheta^m_{\eta_a}(\vartheta_{\eta_a} + b_i^{(a)})\rho^a_0 = \vartheta^m_{\eta_a}(b_i^{(a)} {}_i\Theta^{(a)} \rho^a_0) = b_i^{(a)} {}_i\Theta^{(a)} \rho^a_m,
\]
and
\[
(\vartheta_{\eta_a} + b_i^{(a)})\rho^k_m = b_i^{(a)} \rho^k_m = b_i^{(a)}  {}_i\Theta^{(a)} \rho^k_m.
\]

From these results, we see how the operator acts on the moment matrix:
\[
(\vartheta_{\eta_a} + b_i^{(a)}) \mathscr{M} = \mathscr{M} (\Lambda^{(a)})^\top + b_i^{(a)} \mathscr{M} = b_i^{(a)}  {}_i\Theta^{(a)} \mathscr{M}.
\]

To prove \eqref{hypRelM2}, we examine how the operator \((\vartheta_{\eta_a} + c_j - 1)\) acts on different elements of the moment matrix:
\[
(\vartheta_{\eta_a} + c_j - 1)\rho^a_m = (c_j - 1) \Theta_j \rho^a_m, \quad \text{and} \quad (\vartheta_{\eta_a} + c_j - 1)\rho^k_m = (c_j - 1) \Theta_j \rho^k_m.
\]

Thus, we can write:
\[
\left( \sum_{i=1}^p \vartheta_{\eta_i} + c_j - 1 \right) \rho^a_m = (c_j - 1) \Theta_j \rho^a_m,
\]
and when acting on the moment matrix, we obtain:
\[
\left( \sum_{i=1}^q \vartheta_{\eta_i} + c_j - 1 \right) \mathscr{M} = \mathscr{M} \left( \sum_{l=1}^q \Lambda_l^\top \right) + (c_j - 1) \mathscr{M} = (c_j - 1) \Theta_j \mathscr{M}.
\]
\end{proof}
\begin{defi}
    For $i\in\{1,\dots,M^{(a)}\}$, $j\in\{1,\dots,N\}$ and $a\in\{1,\dots,p\}$, the connection matrices are defined as
    \begin{align}{}_i\omega^{(a)}&=({}_i\Theta^{(a)}H)^{-1}({}_i\Theta^{(a)}\tilde{S})(\Lambda^{(a)}+b_i^{(a)})\tilde{S}^{-1}H,\\
    {}_i\Omega^{(a)}&=S({}_i\Theta^{(a)}S)^{-1}, \\\omega_i&=(\Theta_iH)^{-1}(\Theta_i\tilde{S})(\Lambda^{(q)}+c_j-1)\tilde{S}^{-1}H,\\
    \Omega_i&=S(\Theta_iS)^{-1}.\end{align}
\end{defi}
\begin{teo}
    The connection matrices fulfill the following relations
    \begin{align}
       \label{connRel1} ({}_i\omega^{(a)})^\top&=b_i^{(a)}{}_i\Omega^{(a)}, \\ \label{connRel2} \omega_j^\top&=(c_j-1)\Omega_j.
    \end{align} 
\end{teo}
\begin{proof}
Starting from \eqref{hypRelM1} and using the Gauss--Borel factorization for the moment matrix, we can prove \eqref{connRel1}. We have:
\[
S^{-1} H \tilde{S}^{-\top} (\Lambda^{(a)})^\top + b_i^{(a)} S^{-1} H \tilde{S}^{-\top} = b_i^{(a)} {}_i\Theta^{(a)} (S^{-1} H \tilde{S}^{-\top}).
\]

Multiplying the equation above by \(S\) on the right and by \(\tilde{S}^\top H^{-1}\) on the left, we obtain:
\[
\tilde{S}^\top H^{-1} S^{-1} H \tilde{S}^{-\top} (\Lambda^{(a)})^\top + b_i^{(a)} \tilde{S}^\top H^{-1} S^{-1} H \tilde{S}^{-\top} = b_i^{(a)} \tilde{S}^\top H^{-1} {}_i\Theta^{(a)} (S^{-1} H \tilde{S}^{-\top}).
\]
This simplifies to:
\[
(H \tilde{S}^{-\top}) ((\Lambda^{(a)})^\top + b_i^{(a)}) = b_i^{(a)} {}_i\Theta^{(a)}.
\]
Thus,
\[
(H \tilde{S}^{-\top}) (\Lambda^{(a)})^\top + b_i^{(a)} H \tilde{S}^{-\top} = b_i^{(a)} {}_i\Theta^{(a)} H \tilde{S}^{-\top}.
\]
This confirms \eqref{connRel1}.

Similarly, equation \eqref{connRel2} can be proved starting from \eqref{hypRelM2} using an analogous process:
\[
\left( \sum_{i=1}^q \vartheta_{\eta_i} + c_j - 1 \right) \mathscr{M} = \mathscr{M} \left( \sum_{l=1}^q \Lambda_l^\top \right) + (c_j - 1) \mathscr{M} = (c_j - 1) \Theta_j \mathscr{M}.
\]
Following a similar factorization and multiplication approach as above, we will obtain:
\[
\left( H \tilde{S}^{-\top} \right) \left( \sum_{i=1}^q \Lambda_i^\top + c_j - 1 \right) = (c_j - 1) \Theta_j \left( H \tilde{S}^{-\top} \right).
\]
\end{proof}
\begin{pro}
Connection matrices have a triangular structure with \( p \) non-zero subdiagonals:
\begin{align}
	\label{connEst1} {}_i\omega^{(a)} &= b_i^{(a)}I + ({}_i\omega^{(a)})^{[1]}\Lambda + \dots + ({}_i\omega^{(a)})^{[p]}\Lambda^p, \\
	\label{connEst2}{}_i\Omega^{(a)} &= I + \Lambda^\top ({}_i\Omega^{(a)})^{[1]} + \dots + (\Lambda^\top)^p ({}_i\Omega^{(a)})^{[p]},
\end{align}
with 
\begin{align*}
	({}_i\omega^{(a)})^{[p]} &= I^{(a)} \mathfrak{a}_{-}^p H {}_i\Theta^{(a)} H^{-1}, & ({}_i\Omega^{(a)})^{[m]} &= \sum_{l=0}^m \mathfrak{a}_{-}^{m-l} S^{[l]} {}_i\Theta^{(a)} S^{[l-m]}, & ({}_i\omega^{(a)})^{[m]} &= b_i^{(a)} ({}_i\Omega^{(a)})^{[m]}.
\end{align*}
\end{pro}
\begin{proof}
Equations \eqref{connEst1} and \eqref{connEst2} can be proved starting from \eqref{connRel1} and using the definitions of the connection matrices. We have:
\[
({}_i\omega^{(a)})^\top = H \left(I - \tilde{S}^{[1]}\Lambda + \tilde{S}^{[2]}\Lambda^2 - \dots \right] \left[(\Lambda^q)^\top + b_i^{(a)}\right] {}_i\Theta^{(a)} \left[I + \tilde{S}^{[1]}\Lambda + \tilde{S}^{[2]}\Lambda^2 + \dots \right] H^{-1}.
\]
Expanding this, we get:
\[
({}_i\omega^{(a)})^\top = H \left[(\Lambda^\top)^p I^{(a)} \mathfrak{a}_{-}^p H {}_i\Theta^{(a)} H^{-1} + \dots \right].
\]

For the second connection matrix:
\[
{}_i\Omega^{(a)} = I + \dots + (\Lambda^\top)^m \left[\sum_{l=0}^m \mathfrak{a}_{-}^{m-l} S^{[l]} {}_i\Theta^{(a)} S^{[l-m]}\right].
\]

Therefore, the expanded forms show that:
\[
{}_i\omega^{(a)} = b_i^{(a)} I + ({}_i\omega^{(a)})^{[1]} \Lambda + \dots + ({}_i\omega^{(a)})^{[p]} \Lambda^p,
\]
where 
\[
({}_i\omega^{(a)})^{[p]} = I^{(a)} \mathfrak{a}_{-}^p H {}_i\Theta^{(a)} H^{-1},
\]
and
\[
{}_i\Omega^{(a)} = I + \Lambda^\top ({}_i\Omega^{(a)})^{[1]} + \dots + (\Lambda^\top)^p ({}_i\Omega^{(a)})^{[p]},
\]
where 
\[
({}_i\Omega^{(a)})^{[m]} = \sum_{l=0}^m \mathfrak{a}_{-}^{m-l} S^{[l]} {}_i\Theta^{(a)} S^{[l-m]}.
\]

Additionally, the relationship between these matrices is given by:
\[
({}_i\omega^{(a)})^{[m]} = b_i^{(a)} ({}_i\Omega^{(a)})^{[m]}.
\]
\end{proof}
The type II and type I multiple orthogonal polynomials satisfy:

\begin{pro}
    The following equations hold
    \begin{align}
      \label{co-v:1}{}i\omega^{(a)}A^{(a)}(z)&=(z+b_i^{(a)})({}_i\Theta^{(a)}A^{(a)}(z))\\
      \label{co-v:2}{}_i\Omega^{(a)}{}_i\Theta^{(a)}B(z)&=B(z)\\
      \label{co-v:3}\omega_iA^{(a)}(z)&=(z+c_i-1)(\Theta_iA^{(a)}(z)\\
      \label{co-v:4}\Omega_i\Theta_iB(z)&=B(z).
    \end{align}
\end{pro}
\begin{proof}
Relations \eqref{co-v:1}, \eqref{co-v:2}, \eqref{co-v:3}, and \eqref{co-v:4} are proven by applying the connection matrices to the vectors of polynomials. As an example, the proof of \eqref{co-v:1} is shown below:
 \begin{equation*}
     ({}_i\Theta^{(a)}H^{-1}\tilde{S})(\Lambda^{(a)}+b_i^{(a)})\tilde{S}^{-1}H(H^{-1}\tilde{S}X^{a}(z))={}_i\Theta^{(a)}(H^{-1}\tilde{S})(z+b_i^{(a)})X^{(a)}(z)=(z+b_i^{(a)}){}_i\Theta^{(a)}A^{(a)}(z).
 \end{equation*}
\end{proof}
\section{Multiple Toda systems and hypergeometric $\tau$ functions}
In this section the multiple Toda equations and the multiple Toda system for the recurrence coefficients will be studied. In particular, for the case of three weights Toda equations and the differential-difference equation solved by $tau$ function will be studied.
\begin{defi}
For $a\in\{1,\dots,p\}$,     the consider the strictly lower triangular matrices:
 \[   \begin{aligned}
        \phi^{(a)}&:=(\vartheta^{(a)}S)S^{-1}, & \tilde{\phi}^{(a)}&:=(\vartheta^{(a)}\tilde{S})\tilde{S}^{-1}.
    \end{aligned}\]
\end{defi}
Differential properties of the multiple orthogonal polynomials  can be described in terms of these matrices:
\begin{lemma}
   For $a\in\{1,\dots,p\}$,  the multiple orthogonal polynomials of type II and type I satisfy:
\[    \begin{aligned}
        \vartheta^{(a)}B&=\phi^{(a)}B, &
        \vartheta^{(a)}(HA^{(a)})&=\tilde{\phi}^{(a)}HA^{(a)}.
   \end{aligned}\]
\end{lemma}
\begin{proof}
	First, if we act with $\vartheta^{(a)}$ over $B$ we have the following equalities:
	\begin{equation*}
		\vartheta^{(a)}B=(\vartheta^{(a)}S)X=(\vartheta^{(a)}S)S^{-1}SX=(\vartheta^{(a)}S)S^{-1}B.
	\end{equation*} 
 Second,   we can see that 
    \begin{equation*}\vartheta^{(a)}(HA^{(a)})=(\vartheta^{(a)}\tilde{S})X^{(a)}=\tilde{\phi}^{(a)}HA^{(a)}.
    \end{equation*}
\end{proof}

The interplay between these matrices follow from the Gauss--Borel factorization: 
\begin{pro}
 For $a\in\{1,\dots,p\}$, the following equations are fulfilled
    \begin{align}
       \label{s-w} \vartheta^{(a)}H-\phi^{(a)}H-H(\tilde{\phi}^{(a)})^\top&=(T^{(a)})^\top H, 
    \end{align}
\end{pro}
\begin{proof}
    Departing from the Gauss--Borel factorization for the moment matrix we can see that
    \begin{equation*}
        \vartheta^{(a)}(S^{-1}H\tilde{S}^-\top)=-S^{-1}(\vartheta^{(a)}S)S^{-1}H\tilde{S}^{-\top}+S^{-1}(\vartheta^{(a)}H)\tilde{S}^{-\top}-S^{-1}H\tilde{S}^{-\top}(\vartheta^{(a)}\tilde{S})^\top\tilde{S}^{-\top}=S^{-1}H\tilde{S}^{-\top}(\Lambda^{(a)})^\top.
    \end{equation*}
Consequently, we find:
    \begin{equation*}
        \vartheta^{(a)}H-(\vartheta^{(a)}S)S^{-1}H-H\tilde{S}^{-\top}(\vartheta^{(a)}\tilde{S})^\top=H\tilde{S}^{-\top}(\Lambda^{(a)})^\top\tilde{S}^\top H^{-1}H=(T^{(a)})^\top H.
    \end{equation*}
\end{proof}
For  the complete derivation $\vartheta$ we have the following associated strictly lower matrices:
\begin{defi}
    The following matrices are defined
    \begin{align}
        \phi&:=\sum^p_{a=1}\phi^{(a)}, & \tilde{\phi}&:=\sum^p_{a=1}\tilde{\phi}^{(a)}.
    \end{align}
\end{defi}
\begin{pro}
    The following equations are fulfilled
    \begin{align}
        \label{S-W:d}(\vartheta H)H^{-1}&=\alpha^{(0)},\\
        \label{S-W:sub}-\phi&=\sum_{a=1}^p(\Lambda^\top)^a\alpha^{(a)},\\
        \label{S-W:super}-\tilde{\phi}&=\Lambda^\top\mathfrak a_{-}HH^{-1}.
    \end{align}
\end{pro}
\begin{proof}
    Adding equations \eqref{s-w} for  $a\in\{1,\dots,p\}$ we obtain 
    \begin{equation}\vartheta H-\phi H-H(\phi)^\top=(T)^\top H. \end{equation}
Identifying diagonal, subdiagonal and superdiagonal parts in both sides of the equation we obtain equations \eqref{S-W:d}, \eqref{S-W:sub} and \eqref{S-W:super}. 
\end{proof}
For the subdiagonal in the lower connection factors in the Gauss--Borel factorization we have: 
\begin{pro}
     The following equations are satisfied
     \begin{align}
         \label{tS1}\vartheta\tilde{S}^{[1]}&=-\mathfrak a_{-}HH^{-1},\\
         \label{tSn}\vartheta\tilde{S}^{(a)}&=-\mathfrak a_{-}^aH\mathfrak a_{-}^{a-1}H^{-1}\tilde{S}^{[a-1]},\\
         \label{S1}\vartheta S^{[1]}&=-\alpha^{(1)},\\
         \label{Sn}\vartheta S^{[a]}&=-\alpha^{(a)}-\sum^{a-1}_{i=1}\mathfrak a_{-}\alpha^{(a-i)}S^{[i]}.
     \end{align}
\end{pro}
\begin{proof}
Equations \eqref{tS1} and \eqref{tSn} can be proved equating diagonal by diagonal \eqref{S-W:super}. Analogously, equations \eqref{S1} and \eqref{Sn} can be proved from \eqref{S-W:sub}.
\end{proof}
\begin{pro}
    The following equation are fulfilled
\[    \begin{aligned}
        \vartheta\tilde{S}^{[-a]}&=\mathfrak a_{-}\tilde{S}^{[1-a]}\mathfrak a_{-}HH^{-1}, &  a&\in\{1,\dots,q\}.
    \end{aligned}\]
\end{pro}
\begin{proof}
    Departing from $\vartheta(\tilde{S}\tilde{S}^{-1})=\vartheta(I)=0$ we have that $0=(\vartheta \tilde{S})\tilde{S}^{-1}+\tilde{S}(\vartheta\tilde{S}^{-1})$, so:
   \begin{equation*}
        \tilde{S}\vartheta\tilde{S}^{-1}=-\Lambda^\top\mathfrak a_{-}H H^{-1}
    \end{equation*}
and we obtain that
\begin{equation*}
    \vartheta\tilde{S}^{-1}=-\tilde{S}^{-1}\Lambda^\top\mathfrak a_{-}H H^{-1}.
\end{equation*}   
\end{proof}
\begin{pro}    
For $a\in\{1,\dots,p\}$, the recursion coefficients can be expressed as follows
\begin{align}
    \label{relRec0}\alpha^{(0)}&=(\vartheta H)H^{-1},\\
    \label{relRecN}\alpha^{(a)}&=-[\vartheta S^{[a]}+\sum_{j=1}^{a-1}(\vartheta \mathfrak a_{-}^jS^{[a-j]})S^{[-j]}].
\end{align}
\end{pro}
\begin{proof}
    It is proven splitting Equation \eqref{S-W:sub} by diagonals.
\end{proof}

We now derive a differential system for the functions
\[
\begin{aligned}
	  q_n&\coloneq \log{H_n}, & f_n^{(m)}&\coloneq S^{[m]}_n,
\end{aligned}
\]
 where $m\in\{1,\dots,p\}$ 
\begin{teo}\label{theo:third_order_PDE}
    Functions $q_n=\log{H_n}$, $f_n^{(m)}=S^{[m]}_n$, where $m\in\{1,\dots,p\}$ solve the following system:
    \begin{align}
        \label{MTS-0} \vartheta q_n&=f_{n-1}^{(1)}-f_n^{(1)},\\
        \label{MTS-n}f^{(m+1)}_{n-1}+f^{(-m-1)}_n+\sum_{j=1}^mf^{(m-j)}_{n+j}f^{(-j)}_{n}&=-[\vartheta f^{(m)}_n+\sum^{m-1}_{j=1}(\vartheta f^{(m-j)}_{m+j})f^{(-j)}_m],\\
        \label{MTS-q}e^{q_{n+p}-q_n}&=-[\vartheta f_n^{(p)}+\sum^{p-1}_{j=1}(\vartheta f^{(p-j)}_{n+j})f^{(-j)}_n].
    \end{align}
    
\end{teo}
\begin{proof}
Equations \eqref{MTS-0} and \eqref{MTS-n} are obtained from \eqref{relRec0} and \eqref{relRecN} for $m\in\{1,\dots,p-2\}$ and equating them to expressions of the coefficients $\alpha$ depending on $S$ matrices.

Equation \eqref{MTS-q} is obtained from \eqref{MTS-n} for $m=p-1$ and \eqref{MTS-n} for $m=p$: we derive the first equation, isolate $\vartheta S^{[p]}$ in the second equation and substitute in order to eliminate $S^{[p]}.$
\end{proof}

As an example of the result above we have the case for three weights. Now, we have that $q_n=\log{H_n}$, $f_n=S^{[1]}_n$, $g_n=S^{[2]}_n$ and $k_n=S^{[3]}_n.$\\
We have four equations:
\begin{align}
    \label{3p0}\vartheta q_n&=f_{n-1}-f_n,\\
    \label{3p1}\vartheta f_n&=g_n-g_{n-1}+f_n(f_n-f_{n+1}),\\
    \label{3p2}\vartheta g_n - (\vartheta f_{n+1})
    f_n&=k_n-g_{n+1}f_n-f_{n+2}g_n+f_nf_{n+1}f_{n+2}-k_{n-1}+f_nf_{n+2}+g_n-f_nf_{n+1},\\
    \label{3p3}e^{q_{n+3}-q_n}&=-\vartheta k_n+\vartheta g_{n+1}f_n+\vartheta f_{n+2}(g_n-f_nf_{n+1}).
\end{align}
The variable $k_n$ can be omited if we derivate equation \eqref{3p2} and obtain $\vartheta k_n$ from \eqref{3p3}, so the resultant expression is:
\begin{multline}
    \label{3p2p3}\vartheta^2g_n-(\vartheta^2f_{n+1})f_n-(\vartheta f_{n+1})\vartheta f_n=e^{q_{n+2}-q_{n-1}}-e^{q_{n+3}-q_n}-(\vartheta f_n)g_{n+1}-f_{n+2}\vartheta g_n+(\vartheta f_n) f_{n+1}f_{n+2}+f_n(\vartheta f_{n+1})f_{n+2}\\-(\vartheta g_n )f_{n-1}+
(    \vartheta f_{n+1})(f_{n-1}f_n-g_{n-1})+(\vartheta f_n) f_{n+2}+\vartheta g_n+f_n\vartheta f_{n+2}-f_{n+1}\vartheta f_n-f_n\vartheta f_{n+1}.
\end{multline}
In terms of $\tau$-functions we have that:
\begin{align}
    \label{tauq}q_n&=\log{H_n}=\log{\tau_{n+1}}-\log{\tau_n},\\
    \label{tauf}f_n&=-\frac{\vartheta\tau_{n+1}}{\tau_{n+1}},\\
    \label{taug}g_n&=\frac{\tau_{n+2}^2}{\tau_{n+2}}.
\end{align}
If \eqref{3p1} is written in terms of $\tau$-functions, the following result is obtained:
\begin{equation}
    g_n-g_{n-1}=\frac{\tau_{n+2}^2}{\tau_n+2}-\frac{\tau_{n+1}^2}{\tau_{n+1}}=\frac{(\vartheta\tau_{n+1})\vartheta\tau_{n+2}}{\tau_{n+1}\tau_{n+2}}-\frac{\vartheta^2\tau_{n+1}}{\tau_{n+1}},
\end{equation}
so we have that
\begin{equation}
    \label{gn}g_n=g_0+\sum_{j=1}^n\left(\frac{(\vartheta\tau_{j+1})\vartheta\tau_{j+2}}{\tau_{j+1}\tau_{j+2}}-\frac{\vartheta^2\tau_{j+1}}{\tau_{j+1}}\right).
\end{equation}
Let us take a derivation on  Equation \eqref{gn}:
\begin{equation}
    \label{dergn}\vartheta g_n=\vartheta g_0+\sum_{j=1}^n\left(\frac{(\vartheta^2\tau_{j+1})\vartheta\tau_{j+2}+(\vartheta\tau_{j+1})\vartheta^2\tau_{j+2}}{\tau_{j+1}\tau_{j+2}}-\frac{(\vartheta\tau_{j+1})^2\vartheta\tau_{j+2}}{\tau_{j+2}(\tau_{j+1})^2}-\frac{(\vartheta\tau_{j+2})^2\vartheta\tau_{j+1}}{\tau_{j+1}(\tau_{j+2})^2}-\frac{\vartheta^3\tau_{j+1}}{\tau_{j+1}}+\frac{(\vartheta^2\tau_{j+1})\vartheta\tau_{j+1}}{\tau_{j+1}}\right).
\end{equation}
\begin{rem}
By substituting the expressions from \eqref{tauq}, \eqref{tauf}, \eqref{gn}, and \eqref{dergn} into equation \eqref{3p2p3}, we obtain a fourth-order differential equation (with discrete integrations) for the \(\tau\)-functions.
\end{rem}

\begin{pro}
For  $n\in\{0,\dots,p-1\}$, we have the following Toda type system
 \[   \begin{aligned}
        \vartheta\alpha^{(0)}&=\alpha^{(1)}-\mathfrak a_{+}\alpha^{(1)},\\
        \vartheta\alpha^{(n)}&=\alpha^{(n+1)}-\mathfrak a_{+}\alpha^{(n+1)}+\alpha^{(n)}(\mathfrak a_{-}^n\alpha^{(0)}-\alpha^{(0)}),\\
        \vartheta\alpha^{(p)}&=\alpha^{(p)}(\mathfrak a_{-}^p\alpha^{(0)}-\alpha^{(0)}),
    \end{aligned}\]
\end{pro}
\begin{proof}
This can be proved by applying the operator \(\vartheta\) to Equations \eqref{relRec0} and \eqref{relRecN}, and utilizing Equation \eqref{S-W:d}.
\end{proof}
\begin{pro}
    The following Lax equation is satisfied:
    \begin{equation}
        \vartheta T=[\phi,T]=[T_+,T],
    \end{equation}
    where $T_{-}=\sum^{p}_{n=1}(\Lambda^\top)^n\alpha^{(n)}$ and $T_{+}\alpha^{(0)}+\Lambda.$
\end{pro}
\begin{proof}
    Departing from $T=S\Lambda S^{-1}$ we have that
    \begin{equation}
        \vartheta T=(\vartheta S)\Lambda S^{-1}-S\Lambda S^{-1}(\vartheta S)S^{-1}=\phi T-T\phi=[\phi,T],
    \end{equation}
    and attending to Equation \eqref{S-W:sub} we  see that $\phi=-T_{-}=T_{+}-T$ so $[\phi,T]=[T_{+},T].$
\end{proof}
\begin{pro}
The following compatibility conditions hold:
\begin{align}
	\label{comp2a} \vartheta \Psi &= [\phi, \Psi], \\
	\label{comp2b} \vartheta \Psi^\top &= [\mu, \Psi^\top] + \Psi^\top,
\end{align}
where \(\mu^{(n)} := (\vartheta_n H^{-1}) H + H^{-1} \tilde{\phi}^{(n)} H\) and \(\mu = \sum_{i=1}^p \mu^{(i)}\).
\end{pro}
\begin{proof}
   Starting from equations 
\[   \begin{aligned}
   	\theta(z)B(z-1)&=\Psi B(z), & \vartheta B(z)&=\phi B(z),
   \end{aligned} \]
   we find:
   \begin{align*}
   	\vartheta(\theta(z)B(z-1)) &= \vartheta(\Psi B(z)) = \vartheta\Psi B(z) + \Psi\vartheta B(z) = (\vartheta\Psi + \Psi\phi) B(z),\\
   	\vartheta(\theta(z)B(z-1)) &= \theta(z) \phi B(z-1) = \phi\Psi B(z).
   \end{align*}
   Equating both equations, we have:
   \[
   \phi\Psi = \vartheta\Psi + \Psi\phi,
   \]
   and isolating the term with the derivative of \(\Psi\), we obtain \eqref{comp2a}.
   
   To derive \eqref{comp2b}, we first prove that \(\vartheta^{(m)}A^{(n)} = \mu^{(m)}A^{(n)}\):
   \begin{align*}
   	\vartheta^{(m)}A^{(n)} &= \vartheta^{(m)}(H^{-1}\tilde{S}X^{(n)}) \\
   	&= ((\vartheta^{(m)}H^{-1})HH^{-1}\tilde{S}+H^{-1}(\vartheta^{(m)}\tilde{S})(H^{-1}\tilde{S})^{-1}(H^{-1}\tilde{S}))X^{(n)} \\
   	&= ((\vartheta^{(m)}H^{-1})H + H^{-1}\tilde{\phi}^{(m)}H)A^{(n)} \\
   	&= \mu^{(m)}A^{(n)}.
   \end{align*}
   From this equation, we get \(\sum_{i=1}^p \vartheta^{(i)}A^{(n)} = \sum_{i=1}^q \mu^{(i)}A^{(n)} \rightarrow \vartheta A^{(n)} = \mu A^{(n)}\).
   
   Now, starting from equations 
\[   \begin{aligned}
   	\sigma_n(z)A^{(n)}(z+1) &= \Psi^\top A^{(n)}(z), & \vartheta A^{(n)} &= \mu A^{(n)},
   \end{aligned}\]
   we find:
   \begin{align*}
   	\vartheta(\sigma_n(z)A^{(n)}(z+1)) &= \vartheta(\Psi^\top A^{(n)}) = (\vartheta \Psi^\top)A^{(n)} + \Psi^\top\vartheta A^{(n)} = (\vartheta\Psi^\top + \Psi^\top\mu)A^{(n)},\\
   	\vartheta(\sigma_n(z)A^{(n)}(z+1)) &= \sigma_n(z)(\mu A^{(n)}(z+1) + A^{(n)}(z+1)) = (\mu + I)\Psi^\top)A^{(n)}(z),
   \end{align*}
   so equating both equations and simplifying, we obtain:
   \[
   \vartheta\Psi^\top = [\mu,\Psi^\top] + \Psi^\top.
   \]
\end{proof}

\paragraph{\textbf{Multiple Nijhoff-Capel discrete Toda equations}}
We study now the compatibility between the shifts of the $p+1$ families of parameters for the hypergeometric series:
\begin{equation*}
    \{b_i^{(n)}\}_{i=1}^{M^{(a)}}, \{c_i\}_{i=1}^{N},
\end{equation*}
where $a\in\{1,\dots,p\}.$
For a function $u_n$ it is considered the following notation:
\begin{enumerate}
	\item $\hat u$ means a shift in one and only one  parameter  in the family of parameters $\{b^{(a)}\}$, i. e.  $b^{(a)}_r\to b^{(a)}_r+1$ for only one  given $r\in\{1,\dots,M^{(a)}\}$, the corresponding  connection matrix is denoted by $\hat\Omega^{(r)}$, we denote $\hat d\coloneq b^{(a)}_r$,
	\item  $\check u$  means a shift in one and only one  parameter  in the family of parameters $\{b^{(a)}\}$ with $m\neq m$ , i. e.  $b^{(a)}_s\to b^{(a)}_s+1$ for only one  given $q\in\{1,\dots,M^{(a)}\}$, the corresponding  connection matrix is denoted by $\check \Omega^{(r)}$,  we denote $\check  d\coloneq b^{(m)}_p$,
	\item $\tilde u$ means a shift in one and only one  parameter  in the family of parameters $\{c\}$, i. e.   $c_s\to c_s-1$ for only one  given $s\in\{1,\dots,N\}$, the corresponding  connection matrix is denoted by $\tilde \Omega^{(s)}$, we denote $\tilde  d\coloneq c_s-1$,
	\item $\bar u$ to denote the shift $n\to n+q$.
\end{enumerate}
\begin{lemma}Connection matrices fulfill the following compatibility conditions:
\begin{subequations}
	\begin{align}\label{compsr}\Omega^{(s)}\tilde{\Omega}^{(r)}=\Omega^{(r)}\hat{\Omega}^{(s)}, \\
\label{compsq}\Omega^{(s)}\tilde{\Omega}^{(q)}=\Omega^{(q)}\check{\Omega}^{(s)},\\
\label{comprq}\Omega^{(r)}\hat{\Omega}^{(q)}=\Omega^{(q)}\check{\Omega}^{(r)}.
\end{align}
\end{subequations}
\end{lemma}
\begin{proof}
In the one hand, connection formulas can be written as follows
\[\begin{aligned}\Omega^{(r)}\hat{B}&=B, & \Omega^{(s)}\tilde{B}=B, \end{aligned}
\]so that $\tilde{\Omega}^{(r)}\tilde{\hat{B}}=\tilde{B} $ and, consequently, 
$\Omega^{(s)}\tilde{\Omega}^{(r)}\tilde{\hat{B}}=\Omega^{(s)}\tilde{B}=B$.
On the other hand, we have
$\hat{\Omega}^{(s)}\hat{\tilde{B}}=\hat{B}$ so that 
$\Omega^{(r)}\hat{\Omega}^{(s)}\hat{\tilde{B}}=\Omega^{(r)}\hat{B}=B$. 
Hence, the compatibility $\tilde{\hat{B}}=\hat{\tilde{B}}$ leads to \eqref{compsr}.
Relations \eqref{compsq} and \eqref{comprq} are proven analogously.
\end{proof}
\begin{lemma}
The following equations are satisfied
\begin{subequations}\label{comp}
	 \begin{align}\label{compJo}
 	\hat{T}^\intercal\omega^{(r)}=\omega^{(r)}T^\intercal \\ 
        \label{compJoq}\check{T}^\intercal\omega^{(q)}=\omega^{(q)}T^\intercal \\  
        \label{compJos}\tilde{T}^\intercal\omega^{(s)}=\omega^{(s)}T^\intercal
  \end{align} 
\end{subequations}
 holds.
 \end{lemma}
 \begin{proof} All Equations in \eqref{comp} are proven similarly. Let us prove \eqref{compJo}. From the eigenvalue equation for \(T^\intercal\) and \(A^{(a)}\), \({A^{(a)}}^\intercal T = z{A^{(a)}}^\intercal\), and Equation \eqref{co-v:1} for this case, \(\hat{A}^{(a)}(z) = \frac{\omega^{(r)}}{z+\hat{d}}A^{(a)}(z)\), it can be seen that:
 	\[
 	{{}{\hat A}^{(a)}}^\intercal \hat{T} = z{{}{\hat A}^{(a)}}^\intercal,
 	\]
 	so that:
 	\[
 	\hat{T}^\intercal\frac{\omega^{(r)}}{z+\hat{d}}A^{(a)}(z) = z\frac{\omega^{(r)}}{z+\hat{d}}A^{(a)}(z) = \frac{\omega^{(r)}}{z+\hat{d}}T^\intercal A^{(a)}.
 	\]
 	Hence, if we denote \(M = (\hat{T}^\intercal \omega^{(r)} - \hat{T}^\intercal\omega^{(r)}T)H^{-1}\tilde S\), we find \(MX^{(a)} = 0\), so that \(M = 0\). Since \(H^{-1}\tilde S\) is a lower triangular invertible matrix, we conclude the result. Finally, we obtain:
 	\[
 	\hat{T}^\intercal \omega^{(r)} = \omega^{(r)}T^\intercal.
 	\]
 	For equations \eqref{compJoq} and \eqref{compJos}, similar proofs can be carried out.
 \end{proof}

As an example, for three weights, $p=3$, we find tridiagonal connection matrices:
\begin{align*}
    \Omega^{(s)}&=(\Lambda^\top)^3(\Omega^{(s)})^{[3]}+(\Lambda^\top)^2(\Omega^{(s)})^{[2]}+\Lambda^\top(\Omega^{(s)})^{[1]}+I,\\
    \Omega^{(r)}&=(\Lambda^\top)^3(\Omega^{(r)})^{[3]}+(\Lambda^\top)^2(\Omega^{(r)})^{[2]}+\Lambda^\top(\Omega^{(r)})^{[1]}+I,\\
    \Omega^{(q)}&=(\Lambda^\top)^3(\Omega^{(q)})^{[3]}+(\Lambda^\top)^2(\Omega^{(q)})^{[2]}+\Lambda^\top(\Omega^{(q)})^{[1]}+I,
\end{align*}
where the diagonals are given by
\begin{align*}
    (\Omega^{(s)})^{[3]}&=\frac{\mathfrak a_{-}^3H}{\tilde{d}\tilde{H}}, & (\Omega^{(s)})^{[2]}&=S^{[2]}-\tilde{S}^{[2]}+\tilde{S}^{[1]}(\mathfrak a_{-}\tilde{S}^{[1]}-\mathfrak a_{-}S^{[1]}), &(\Omega^{[s]})^{[1]}&=S^{[1]}-\tilde{S}^{[1]},\end{align*}
\begin{align*}(\Omega^{(r)})^{[3]}&=p_n\frac{\mathfrak a_{-}^3H}{\hat{d}\hat{H}}, & (\Omega^{(r)})^{[2]}&=S^{[2]}-\hat{S}^{[2]}+\hat{S}^{[1]}(\mathfrak a_{-}\hat{S}^{[1]}-\mathfrak a_{-}S^{[1]}), &(\Omega^{[r]})^{[1]}&=S^{[1]}-\hat{S}^{[1]},\end{align*}
\begin{align*}(\Omega^{(q)})^{[3]}&=p_m\frac{\mathfrak a_{-}^3H}{\check{d}\check{H}}, & (\Omega^{(q)})^{[2]}&=S^{[2]}-\check{S}^{[2]}+\check{S}^{[1]}(\mathfrak a_{-}\check{S}^{[1]}-\mathfrak a_{-}S^{[1]}), &(\Omega^{[q]})^{[1]}&=S^{[1]}-\check{S}^{[1]}.
\end{align*}

\section{Laguerre--Freud Equations by Examples}
Our objective is to determine the Laguerre--Freud matrix for various specific cases and derive the coefficients of the polynomials or the Laguerre--Freud equations for these coefficients. 

In the case of \(p\) weights, our extension to this multiple setting of the  Laguerre--Freud equations for the coefficients of the \(T\) matrix are defined as:
\begin{align*}
    \alpha^{(0)}_n&=\mathscr F_n^{(0)}(\alpha^{(p)}_n,\alpha^{(p-1)}_n,\dots,\alpha^{(1)}_n,\alpha^{(p)}_{n-1},\dots),\\
    \alpha^{(m)}_n&=\mathscr F_n^{(m)}(\alpha^{(p)}_n,\dots,\alpha^{(m+1)}_n,\alpha^{(p)}_{n-1},\dots,\alpha^{(0)}_{n-1},\dots),\\
    \alpha^{(p)}_n&=\mathscr F_n^{(p)}(\alpha^{(p)}_{n-1},\dots,\alpha^{(0)}_{n-1},\alpha^{(p)}_{n-2},\dots),
\end{align*}
with $m\in\{1,\dots,p-1\}.$
\subsection{Multiple Charlier case}
In this scenario, the weights are given by \(w^{(a)}(k) = \frac{(\eta^{(a)})^k}{k!}\) with \(a\in \{1,\dots,p\}\), and the polynomials of the Pearson equations are \(\theta = k\) and \(\sigma^{(a)}(k) = \eta^{(a)}\), where \(a \in \{1,\dots,p\}\).

Since  \(\deg \theta=1\), we have one non-zero superdiagonal and zero non-null subdiagonal. This results in the following structure for the Laguerre--Freud matrix:
\begin{equation*}
    \Psi=\psi^{(0)}+\psi^{(1)}\Lambda.
\end{equation*}
By applying \(\Psi = \Pi^{-}\theta(T)\), we can obtain:
\[\begin{aligned}
	\psi^{(0)} &= \alpha^{(0)} - \mathfrak{a}_{+}D, & \psi^{(1)} &= I;
\end{aligned}\]
and using \(\Psi = \sum_{a=1}^p \eta^{(a)} (\Pi^{(a)+})^\top\), we have:
\[\begin{aligned}
	\psi^{(0)} &= \sum_{a=1}^p \eta^{(a)} I^{(a)}, & \psi^{(1)} &= \sum_{a=1}^p \eta^{(a)} H \mathfrak{a}_{-} H^{-1} (I^{(a)} - I^{(a-1)}) \tilde{S}^{[1]}.
\end{aligned}\]
By equating both expressions for \(\psi^{(0)}\), we can derive that:
\[
\begin{aligned}
	\alpha^{(0)}_a &= a + \eta^{(a+1)}, & a&\in \{0,1,\dots,p-1\}.
\end{aligned}
\]
If we calculate the compatibility $[\Psi,T]=\Psi$ we obtain the following non-trivial equations:
\begin{align*}
    \mathfrak a_{-}^{p}\alpha^{(0)}-\alpha^{(0)}&=pI,\\
    \alpha^{(a+1)}-\mathfrak a_{+}\alpha^{(a+1)}&=\alpha^{(a)}(\alpha^{(0)}-\mathfrak a_{-}^{a}\alpha^{(0)}+aI),\\
    \alpha^{(1)}-\mathfrak a_{+}\alpha^{(1)}&=\sum_{a=1}^p\eta^{(a)}I^{(a)},\\
    \mathfrak a_{-}\alpha^{(0)}-\alpha^{(0)}&=\sum^p_{a=1}\eta^{(a)}(\mathfrak a_{-}I^{(a)}-I^{(a)})+I,
\end{align*}
where $a\in\{1,\dots,p-1\}.$
So we can write the coefficients as
\[\begin{aligned}
    \alpha_{pm+a}^{(0)}&=pm+a+\eta^{(a+1)},   &a&\in\{0,\dots,p-1\},\\
    \alpha_{pm+a}^{(1)}&=\sum_{b=1}^a\eta^{(b)}+m\sum^p_{b=1}\eta^{(b)},& a&\in\{1,2,\dots,p-1\},\\
    \alpha_{pm+a}^{(l+1)}&=\alpha_{pm+a-1}^{(l+1)}+\alpha_{pm+a-1}^{(l)}(\eta^{(a-l)}-\eta^{(a)}),&  a&\in
    \{l+1,\dots,p-1\},
\end{aligned}\]
where $m\in\N_0$.
\subsection{Multiple generalized Charlier case}
In this case, the weights are given by
\begin{equation}\label{eq:generalized Charlier}
	w^{(a)}(k) = \frac{1}{(c+1)_k} \frac{(\eta^{(a)})^k}{k!}, 
\end{equation}
and the polynomials from the Pearson equations are
\[\begin{aligned}
	 \theta(k) &= k(k+c), & \sigma^{(a)}(k) &= \eta^{(a)}. 
\end{aligned}\]

Due to the degrees of these polynomials, the Laguerre--Freud matrix \(\Psi\) has the form:
\[ \Psi = \psi^{(0)} + \psi^{(1)}\Lambda + \psi^{(2)}\Lambda^2. \]
Now, using the relation \(\Psi = \Pi^{-}\theta(T)\), we find:
\[
\psi^{(1)} = \alpha^{(0)} + \mathfrak{a}_{-}\alpha^{(0)} + c - \mathfrak{a}_{+}D, \quad \psi^{(2)} = I,
\]
where \(I\) is the identity matrix.

From the equation \(\Psi = \sum_{a=1}^{p} \eta^{(a)} (\pi^{(a)+})^\top\), it follows that:
\[ \psi^{(0)} = \sum_{a=1}^{p} \eta^{(a)} I^{(a)}. \]
From compatibility $[\Psi,T]=\Psi$ we can obtain 

\begin{pro}
	The following Laguerre--Freud equations hold for the generalized multiple Charlier weights \eqref{eq:generalized Charlier}
\begin{align*}
    \alpha^{(0)}_{m+p}&=m+p-1-c-\alpha_{m+p-1}^{(0)}+\frac{\alpha^{(p)}_{m+p-1}}{\alpha^{(p)}_{m+p}}(\alpha_{m-1}^{(0)}+\alpha^{(0)}_{m}+c-m)+\frac{\alpha^{(p-1)}_{m+p-1}}{\alpha^{(p)}_{m+p}}(r(m+1)-r(m)),\\
    \alpha^{(1)}_{m+2}&=\alpha^{(1)}_m+\alpha^{(0)}_{m+1}+\alpha^{(0)}_m+c-m+(\alpha^{(0)}_m-\alpha^{(0)}_{m+1})(\alpha^{(0)}_m+\alpha^{(0)}_{m+1}+c-m)+r(m+2)-r(m+1),\\
    \alpha^{(2)}_{m+2}&=\alpha^{(2)}_m+\alpha^{(1)}_m(\alpha^{(0)}_{m}+\alpha^{(0)}_{m-1}+c-m+1)-\alpha^{(1)}_{m+1}(\alpha^{(0)}_m+\alpha^{(0)}_{m+1}+c-m)+r(m+1),\\
    \alpha^{(n+2)}_{n+2+m}&=\begin{multlined}[t][.8\textwidth]
\alpha^{(n+2)}_{n+m}-\alpha^{(n+1)}_{m+n+1}(\alpha^{(0)}_{a+m}+\alpha^{(0)}_{a+m+1}+c-(a+m))+\alpha^{(a+1)}_{a+m}(\alpha^{(0)}_{m-1}+\alpha^{(0)}_m+c-(m-1))\\+\alpha^{(n)}_{m+a}(r(m+1)-r(a+m+1)),
    \end{multlined}
\end{align*}
where $a\in\{1,\dots,p-2\},$ and $r(n)=n-p\big\lfloor\frac{n}{p}\big\rfloor, \ r(n)\in\{1,\dots,p\}.$ Here $\lfloor x\big\rfloor,$ is the largest integer that is equal or smaller than $x$.
\end{pro}
\subsection{Multiple Meixner II case}
For this case, the weights are given by
\[ w^{(a)}(k) = (b_a)_k \frac{\eta^k}{k!}, \]
and the polynomials from the Pearson equations are
\[ \begin{aligned}
	\theta(k)& = k , &\sigma^{(a)}(k) &= \eta(k + b_a), 
\end{aligned}\]
where \( a \in \{1, \dots, p\} \).

Considering the degrees of these polynomials, the Laguerre--Freud matrix has the structure:
\[ \Psi = (\Lambda^\top)^p \psi^{(-p)} + \cdots + \psi^{(1)} \Lambda. \]
By applying \(\Psi = \Pi^{-}\theta(T)\), we obtain:
\[
\psi^{(1)} = I, \quad \psi^{(0)} = \alpha^{(0)} - \mathfrak{a}_{+} D.
\]
Using the relation \(\Psi = \sum_{a=1}^{p} \sigma_a(T^{(a)})^\top (\pi^{(a)+})^\top\), we find:
\[
\psi^{(0)} = \eta \left( \alpha^{(0)} + \sum_{a=1}^{p} I^{(a)} (b_a + (\mathfrak{a}_{+})^p D^{(a)}) \right), \quad \psi^{(-m)} = \eta \alpha^{(m)}, \quad m \in \{1, \dots, p\}.
\]

From the compatibility condition \([\Psi, T] = \Psi\), we obtain the following non-trivial equations for the coefficient matrices:
\begin{align*}
	\mathfrak{a}_{-} \alpha^{(0)} - \alpha^{(0)} &= \frac{1}{1-\eta} \left( I + \sum_{a=1}^{p} \eta \left( I^{(a+1)} (b_a + \mathfrak{a}_{+}^{p-1} D^{(a)}) - I^{(a)} (b_a + \mathfrak{a}_{+}^p D^{(a)}) \right) \right), \\
	\alpha^{(1)} - \mathfrak{a}_{+} \alpha^{(1)} &= \frac{1}{1-\eta} (\alpha^{(0)} - \mathfrak{a}_{+} D), \\
	\alpha^{(l+1)} - \mathfrak{a}_{+} \alpha^{(l+1)} &= \frac{\eta}{1-\eta} \left( \alpha^{(l)} \left( I + \sum_{a=1}^{p} I^{(a)} (b_a + \mathfrak{a}_{+}^p D^{(a)}) \right) - \sum_{a=1}^{p} I^{(a+l)} (b_a + \mathfrak{a}_{+}^{p-l} D^{(a)}) \right),
\end{align*}
where \( l \in \{1, \dots, p-1\} \).

Equating both expressions for $\psi^{(0)}$ we can obtain $\alpha^{(0)}$ entrywise:
\begin{equation*}
    \alpha^{(0)}_{pa+m}=\frac{1}{1-\eta}(pa+m+\eta(b_{m+1}+a))
\end{equation*}
Departing from compatibility equations and using the expression for the $\alpha^{(0)}$ elements we can obtain:
\begin{align*}
    \alpha^{(1)}_{pa+m}&=\frac{1}{1-\eta}\left(\frac{1}{1-\eta}\left((p+\eta)\frac{a(a+1)}{2}+\frac{m(m+1)}{2}+\eta\sum_{j=0}^mb_{j+1}\right)-\frac{(pa+m-1)(pa+m)}{2}\right), \\
    \alpha^{(l+1)}_{pa+m+l+1}&=\alpha^{(l+1)}_{pa+m+l}+\frac{\eta}{1-\eta}\alpha^{(l)}_{pa+m+l}(1+b_{m+1}+a).
\end{align*}

\subsection{Multiple generalized Meixner II case}
In this case, the weights are given by:
 \begin{equation}\label{eq:generalized Meixner II}
	w^{(a)}(k) = \frac{(b_a)_k}{(c+1)_k} \frac{\eta^k}{k!}, 
\end{equation}
and the polynomials from the Pearson equations are:
\[ \theta(k) = k(k+c) \quad \text{and} \quad \sigma^{(a)}(k) = \eta(k + b_a), \]
where \(a \in \{1, \dots, p\}\).

To find the Laguerre--Freud equation we proceed as follows. Considering the degrees of the  polynomials defining the Pearson equations, the Laguerre--Freud matrix has the structure:
\[ \Psi = (\Lambda^\top)^p \psi^{(-p)} + \cdots + \psi^{(2)} \Lambda^2. \]
Applying \(\Psi = \Pi^{-} \theta(T)\), we obtain:
\[
\psi^{(2)} = I, \quad \psi^{(1)} = \mathfrak{a}_{-} \alpha^{(0)} + \alpha^{(0)} + c - \mathfrak{a}_{+} D.
\]
Using the relation \(\Psi = \sum_{a=1}^{p} \sigma_a(T^{(a)})^\top (\pi^{(a)+})^\top\), we find:
\[
\psi^{(0)} = \eta \left( \alpha^{(0)} + \sum_{a=1}^{p} I^{(a)} (b_a + \mathfrak{a}_{+}^p D^{(a)}) \right), \quad \psi^{(-m)} = \eta \alpha^{(m)}, \quad m \in \{1, \dots, p\}.
\]

From the compatibility condition \([\Psi, T] = \Psi\), we find
\begin{pro} The following Laguerre--Freud equations are satisfied for the generalized multiple Meixner weights  given in \eqref{eq:generalized Meixner II}
	\begin{align*}
	\alpha^{(0)}_{(a+1)p+m}&=\begin{multlined}[t][.8\textwidth]-\alpha^{(0)}_{(a+1)p+m-1}+a+p-1-c+(\alpha^{(p)}_{(a+1)p+m})^{-1}(\eta(\alpha^{(p-1)}_{(a+1)p+m-1}(b_{m+1}+a+1)+\alpha^{(p)}_{(a+1)p+m}\\-\alpha^{(p)}_{(a+1)p+m-1})+\alpha^{(p)}_{(a+1)p+m-1}(\alpha^{(0)}_{pa+m}+\alpha^{(0)}_{pa+m-1}+c-(a-1))),
\end{multlined}\\
	\alpha^{(1)}_{pa+m+2}&=\begin{multlined}[t][.8\textwidth]\alpha^{(1)}_{pa+m}+\eta(\alpha^{(0)}_{pa+m+1}-\alpha^{(0)}_{pa+m}-b_m-a)+(\alpha^{(0)}_{pa+m}-\alpha^{(0)}_{pa+m+1}+1)(\alpha^{(0)}_{pa+m+1}+\alpha^{(0)}_{pa+m}+c-(m+pa)),   
\end{multlined}\\
	\alpha^{(l+2)}_{pa+m+l+2}&=\begin{multlined}[t][.8\textwidth]\alpha^{(l+2)}_{pa+m+l}+\eta(\alpha^{(l+1)}_{pa+m+l+1}-\alpha^{(l+1)}_{pa+m+l})-\alpha^{(l+1)}_{pa+m+l+1}(\alpha^{(0)}_{pa+m+l+1}+\alpha^{(0)}_{pa+m+l}+c-(a+l))\\+\alpha^{(l+1)}_{pa+m+l}(\alpha^{(0)}_{pa+m}+\alpha^{(0)}_{pa+m-1}+c-(a-1)),
\end{multlined}\\
	\alpha^{(l+2)}_{pa+m+l+2}&=\begin{multlined}[t][.8\textwidth]\alpha^{(l+2)}_{pa+m+l}+\eta(\alpha^{(l)}_{pa+m+l}(b_{m+1}+a+1)+\alpha^{(l+1)}_{pa+m+l+1}-\alpha^{(l+1)}_{pa+m+l})+\alpha^{(l+1)}_{pa+m+l+1}(\alpha^{(0)}_{pa+m+l+1}+c-(a+l)\\+\alpha^{(0)}_{pa+m+l})+\alpha^{(l+1)}_{pa+m+l}(\alpha^{(0)}_{pa+m}+\alpha^{(0)}_{pa+m-1}+c-(a-1)).
\end{multlined}
\end{align*}
\end{pro}

\section*{Conclusion and Outlook}
In this paper, we continue our study of semiclassical multiple discrete orthogonal polynomials, now considering an arbitrary number of weights. We have derived multiple versions of the Toda equations and the Laguerre--Freud equations for the multiple generalized Charlier and multiple generalized Meixner II families.

The investigation of mixed multiple discrete orthogonal polynomials, along with the corresponding Toda and Laguerre-Freud equations, presents an intriguing area of research. Additionally, the exploration of Pearson equations in the multiple scenario for the unit circle appears to be both unexplored and promising.

\section*{Acknowledgments}
The authors  acknowledge research project [PID2021- 122154NB-I00], \emph{Ortogonalidad y Aproximación con Aplicaciones en Machine Learning y Teoría de la Probabilidad}  funded  by \href{https://doi.org/10.13039/501100011033}{MICIU/AEI/10.13039/501100011033} and by "ERDF A Way of making Europe”.

\section*{Declarations}

\begin{enumerate}
	\item \textbf{Conflict of interest:} The authors declare no conflict of interest.
	\item \textbf{Ethical approval:} Not applicable.
	\item \textbf{Contributions:} All the authors have contribute equally.
	\item \textbf{Data availability:} This paper has no associated data.
\end{enumerate}

\end{document}